\documentclass[11pt]{amsart}
\usepackage{amssymb,amsmath,amsthm,amsfonts,mathrsfs}
\usepackage[hmargin=3cm,vmargin=3.5cm]{geometry}
\usepackage[dvipsnames,table,xcdraw]{xcolor}
\usepackage[colorlinks = true,
            linkcolor = Fuchsia,
            urlcolor  = ForestGreen,
            citecolor = WildStrawberry,
            anchorcolor = blue]{hyperref}

\usepackage[all,2cell]{xy} \UseAllTwocells 
\usepackage{tikz-cd}

\usepackage{color}
\usepackage{graphicx,psfrag}
\usepackage{amscd}  
\usepackage{stmaryrd}  
\usepackage[all,2cell]{xy} \UseAllTwocells \SilentMatrices
\usepackage{tikz}
\usepackage[dvips]{epsfig}
\usepackage{comment}
\usepackage{kbordermatrix}
\usepackage{cancel}

\theoremstyle{definition}
\newtheorem{theorem}{Theorem}[section]

\newtheorem{remark}[theorem]{Remark}

\newtheorem{prop}[theorem]{Proposition}

\newtheorem{corollary}{Corollary}

\usetikzlibrary{arrows,automata}
\usetikzlibrary{decorations.markings}
\usetikzlibrary{decorations.pathmorphing,shapes}
\usepgflibrary {shadings}

\newcounter{sarrow}

\title{Foam cobordism and the Sah-Arnoux-Fathi invariant}

\author{Mee Seong Im}
\address{Department of Mathematics, United States Naval Academy, Annapolis, MD 21402, USA}
\email{\href{mailto:meeseongim@gmail.com}{meeseongim@gmail.com}}

\author{Mikhail Khovanov} 
\address{Department of Mathematics, Columbia University, New York, NY 10027, USA}
\email{\href{mailto:khovanov@math.columbia.edu}{khovanov@math.columbia.edu}}

\address{Department of Mathematics, Johns Hopkins University, Baltimore, MD 21218, USA}
\email{\href{mailto:khovanov@jhu.edu}{khovanov@jhu.edu}}

\makeatletter
\@namedef{subjclassname@2020}{%
  \textup{2020} Mathematics Subject Classification}


\date{March 9, 2024}

\providecommand{\keywords}[1]{\textbf{\textit{Key words and phrases.}} #1}


\begin{document}

\def\front{\mathsf{front}}
\def\back{\mathsf{back}}

\def\LHS{\mathsf{LHS}}
\def\RHS{\mathsf{RHS}}
\def\up{\mathsf{up}}
\def\E{\mathsf E}
\def\I{\mathsf I}
\def\R{\mathbb R}
\def\Q{\mathbb Q}
\def\Z{\mathbb Z}
\def\N{\mathbb N} 
\def\C{\mathbb C}
\def\S{\mathbb S}
\def\SS{\mathbb S} 
\def\GL{\mathsf{GL}} 

\newcommand{\dmod}{\mathsf{-mod}}
\newcommand{\comp}{\mathrm{comp}} 
\newcommand{\col}{\mathrm{col}}
\newcommand{\adm}{\mathrm{adm}}  
\newcommand{\Ob}{\mathrm{Ob}}
\newcommand{\Cob}
{\mathsf{Cob}}
\newcommand{\ECob}{\mathsf{ECob}}
\newcommand{\indexw}{\R_{>0}}
\newcommand{\indexww}{>0}
\newcommand{\CoboR}{\Cob^1_{\indexww}}
\newcommand{\CoboRz}
{\Cob^0_{\indexww}}
\newcommand{\CoboRf}
{\Cob^{1,f}_{\indexww}}
\newcommand{\id}{\mathsf{id}}
\newcommand{\undM}{\underline{M}}
\newcommand{\im}{\mathsf{im}}
\newcommand{\coker}{\mathsf{coker}}
\newcommand{\Aut}{\mathsf{Aut}}
\newcommand{\tripod}{\mathsf{Td}}
\newcommand{\scrX}{\mathscr{X}}
\newcommand{\scrY}{\mathscr{Y}}
\newcommand{\rmH}{\mathrm{H}}
\newcommand{\disk}{\mathbb{D}^2}
\newcommand{\Ztwo}{Z^2}  

\def\l{\lbrace}
\def\r{\rbrace}
\def\o{\otimes}
\def\lra{\longrightarrow}
\def\lla{\longleftarrow}
\def\Ext{\mathsf{Ext}}
\def\ker{\mathsf{ker}}
\def\mf{\mathfrak} 
\def\mcC{\mathcal{C}}
\def\mcS{\mathcal{S}}  
\def\mcQC{\mathcal{QC}}
\def\mcA{\mathcal{A}}
\def\mcE{\mathcal{E}}
\def\Fr{\mathsf{Fr}}  

\def\bbn{\mathbb{B}^n}
\def\ovb{\overline{b}}
\def\tr{{\sf tr}} 
\def\det{{\sf det }} 
\def\one{\mathbf{1}}   
\def\kk{\mathbf{k}}  
\def\gdim{\mathsf{gdim}}  
\def\rk{\mathsf{rk}}
\def\IET{\mathsf{IET}}
\def\SAF{\mathsf{SAF}}
\newcommand\vspsm{\vspace{0.07in}}

\newcommand{\brak}[1]{\ensuremath{\left\langle #1\right\rangle}}
\newcommand{\oplusop}[1]{{\mathop{\oplus}\limits_{#1}}}
\newcommand{\addfigure}{\vspace{0.1in} \begin{center} {\color{red} ADD FIGURE} \end{center} \vspace{0.1in} }
\newcommand{\add}[1]{\vspace{0.1in} \begin{center} {\color{red} ADD FIGURE #1} \end{center} \vspace{0.1in} }
\newcommand{\vspin}{\vspace{0.1in} }

\newcommand\circled[1]{\tikz[baseline=(char.base)]{\node[shape=circle,draw,inner sep=1pt] (char) {${#1}$};}} 

\let\oldemptyset\emptyset
\let\emptyset\varnothing

\let\oldtocsection=\tocsection
\let\oldtocsubsection=\tocsubsection
\renewcommand{\tocsection}[2]{\hspace{0em}\oldtocsection{#1}{#2}}
\renewcommand{\tocsubsection}[2]{\hspace{1em}\oldtocsubsection{#1}{#2}}

\def\MK#1{{\color{red}[MK: #1]}}
\def\bfred#1{{\color{red}#1}}


\begin{abstract} This is the first in a series of papers where scissor congruence and K-theoretical invariants are related to cobordism groups of foams in various dimensions. A model example is provided where the cobordism group of weighted one-foams is identified, via the Sah--Arnoux--Fathi invariant, with the first homology of the group of interval exchange automorphisms and with the Zakharevich first K-group of the corresponding assembler. Several variations on this cobordism group are computed as well.
\end{abstract}

\maketitle

\tableofcontents

%
%
    
\section{Introduction}
\label{sec_intro} 

In link homology by a \emph{foam} one usually means a 2-dimensional finite combinatorial CW-complex $F$, often embedded in $\R^3$, where each point has one of the three types of neighborhoods shown in Figure~\ref{fig3_006} below.  Foams are used in algebraically-defined link homology to build state spaces of planar graphs which are then combined into complexes that define homology of a link~\cite{Kh1,MV1,RWd,RW1,KK}. Foams also appear in Kronheimer--Mrowka instanton Floer homology for 3-orbifolds~\cite{KrMr}. 

Locally, the foam structure is that of a two-dimensional spine of a 3-manifold. 
Often, foams come 
with extra decorations, such as orientations, weights and other labels on facets.

In this paper a \emph{closed two-foam} means a foam as above, with additional decorations specified. More generally, one can define a \emph{two-foam with boundary}, the boundary being a one-foam. A \emph{one-foam} is a finite graph, possibly with loops and circle edges without vertices, and additional decorations. Splitting the boundary of a two-foam $F$ into two disjoint sets of components, $\partial F = (- U_0)\sqcup U_1$, allows to view $F$ as a cobordism between one-foams $U_0$ and $U_1$. 
Decorations of $U_0,U_1$ are induced from those of $F$.  

\vspace{0.07in} 

This paper is the first in a series of papers which aim to use foams, in all dimensions $n$ and with additional decorations, to understand K-theoretical structures. One expects that $n$-dimensional foams decorated by objects and morphisms of an exact category $\mcC$, modulo concordances which are $\mcC$-decorated $(n+1)$-dimensional foams, carry information about the $n$-th K-theory group $K_n(\mcC)$ of $\mcC$. 
Facets, respectively, seams of a foam are decorated by flat connections with objects of $\mcC$, respectively short exact sequences of $\mcC$,  as fibers of these flat bundles. This relation between decorated foams and algebraic K-theory will be studied in forthcoming papers.

\vspsm

The present paper works out a straightforward example of this correspondence,   
where the abelian group of suitably decorated one-dimensional foams modulo 2-dimensional cobordisms is identified with the group $\R\wedge_{\Q} \R$, which is the first homology of the group of interval exchange transformations~\cite{Vor11}.  The related invariant of interval exchange transformations mapping a group element to its image in the first homology is known as the Sah--Arnoux--Fathi invariant, or $\SAF$ invariant, for short~\cite{Ve,DS16}.  I.~Zakharevich interpreted the $\SAF$ invariant map via the $K_1$ group of a suitable \emph{assembler} category~\cite{Za1,Za2}, and that category plays the role of the exact category $\mcC$ above. 
In the present paper, we relate these structures to two-dimensional cobordisms between decorated one-dimensional foams. 

In Section~\ref{sec_iet} we work out this new interpretation of the $\SAF$ invariant, as classifying elements of the cobordism group of \emph{weighted} oriented one-foams.  In this construction, edges of an oriented one-foam are decorated by positive real numbers $a$, with compatibility relations on these numbers at vertices. The cobordism group of such foams is identified with the abelianization of the group of interval exchange transformations ($\IET$s) in Theorem~\ref{thm_saf}. The isomorphism uses the Sah--Arnoux--Fathi invariant of $\IET$s, extended to arbitrary weighted oriented one-foams. 

 Section~\ref{subsec_planar} considers the cobordism group of planar unoriented weighted 1-foams  and identifies it with the abelian group generated  by brackets $[a,b]$ with $a,b\in \R_{>0}$ modulo the antisymmetry and the 2-cocycle relations \eqref{eq_rel_a}-\eqref{eq_rel_c}. It also looks at a variation on weighted embedded foams, where each facet may carry either a positive or a negative weight. 
Several other variations on the group of foam cobordisms are studied in Section~\ref{sec_variations}.
Constructions of  Sections~\ref{sec_iet} and~\ref{subsec_planar} can perhaps be viewed as first steps exploring the relation between foam cobordisms and dynamical systems. 
   
\vspace{0.1in}  
  
{\bf Acknowledgments:} The authors are grateful to David Gepner, Nitu Kitchloo, and Inna Zakharevich for interesting discussions. M.K. would like to acknowledge partial support from NSF grant DMS-2204033 and Simons Collaboration Award 994328.

%
%

\section{Foams and interval exchange transformations} 
\label{sec_iet}

In this section we interpret the Sah--Arnoux--Fathi invariant of interval exchange transformations~\cite{Ve,Vor11,DS16}  via cobordism classes of oriented one-foams with facets decorated by positive real numbers (called \emph{weighted} or \emph{$\R_{>0}$-decorated} one-foams). 

\vspace{0.07in} 


\subsection{Oriented 1-foams and 2-foams and cobordisms between 1-foams}
\label{subsec_oriented} 

\quad

In this paper, a \emph{closed 2-foam} denotes a finite combinatorial CW-complex $F$, where each point is one of the three types and has a neighborhood as depicted in Figure~\ref{fig3_006}. These points are called \emph{regular points}, \emph{seam points} and \emph{vertices} of the 2-foam, respectively. The union of seams and vertices of $F$ is a four-valent graph $s(F)$, possibly with loops and verticeless circles. 
Connected components of $F\setminus s(F)$ are called \emph{facets} of $F$, and $s(F)$ is called the set of \emph{singular points} of $F$.

\vspsm 

\input{fig3_006}

\input{fig3_005}

A closed 2-foam is \emph{oriented} if 
\begin{itemize}
\item Each facet is oriented, so that along seams and near vertices orientations match as shown in Figure~\ref{fig3_005} (for seams) and Figure~\ref{fig3_008} on the right (for vertices). Along each seam, two of the facets are designated as \emph{thin} and the remaining one as \emph{thick}. Orientation of each thin facet matches (flows into) the orientation of the thick facet. Orientations of the two thin facets along a seam are opposite. 
A facet which is thin at one of its seams may be thick at another seam.
\item   An order of two thin facets along each seam is fixed (shown by small curly arrow from one thin facet to the other in Figure~\ref{fig3_005}. 
\item At each vertex of the foam,  decorations (orientations and orders) of the six adjoint facets along the four seams match as follows (and shown in Figure~\ref{fig3_008} on the right). The six facets are labelled $f_1,f_2,f_3,f_{12},f_{13},f_{123}$. Among the triples of facets $(f_1,f_2,f_{12}), (f_2,f_3,f_{23}), (f_{12},f_3,f_{123}), (f_1,f_{23},f_{123})$, one triple for each seam, the first two facets are thin and the last one is thick. The facets are oriented either as shown in Figure~\ref{fig3_008} on the right or with all orientations opposite (which follows from the orientation requirements along the seams). Orders of facets along the seams are as shown in Figure~\ref{fig3_008} on the left, in the direction of increasing indices, or the opposite (decreasing indices). 
\end{itemize} 

\input{fig3_008}

\input{fig3_009}

Figure~\ref{fig3_009} shows a set of three ``parallel cross-sections'' of a foam near a vertex, with one of the cross-sections going through the vertex. Figure~\ref{fig9_006} depicts a neighbourhood of a vertex taking ``tangencies'' of thin facets along the four seams near the vertex into account, analogous to that of a vertex in a branched surface~\cite[Figure 1.1]{Oer88}, see also~\cite{QW22}. (For now, ignore the weights of facets in Figure~\ref{fig9_006}.)

\vspsm 

\input{fig9_006}

\vspsm 

One-foams can be thought of as generic cross-sections of two-foams. A one-foam is a finite oriented trivalent graph. At each vertex there are two \emph{in} edges and one \emph{out} edge or vice versa. We call these \emph{merge} and \emph{split} vertices, correspondingly. They are shown in Figure~\ref{fig6_001} on the left (ignoring the weights $a,b,a+b$ in that figure). An oriented circle with no vertices on it is allowed as a component of a one-foam. Loops are allowed, although we will not encounter them due to working with \emph{weighted} foams only.  

It is convenient to visualize thin edges at a merge vertex as sharing a tangent line at a vertex and think of a neighbourhood of a merge vertex as a generic cross-section across the seam of the Figure~\ref{fig3_005} foam.  Likewise, a neighbourhood of a split vertex of a 1-foam can be visualized as a horizontal cross-section of the rightmost foam in Figure~\ref{fig6_001}. 
 Similar conventions are used in~\cite{RW2, QW22}.

\vspsm 

We define an oriented  two-foam $F$ with boundary as a cobordism between oriented one-foams $U_0,U_1$. The boundary of $F$ is split into two disjoint 1-foams,  
\[\partial F \cong (-\partial_0 F) \sqcup \partial_1 F \cong
U_1\sqcup (-U_0).
\]
Away from the boundary $F$ has  local structure that of an oriented 2-foam and collar neighbourhoods 
near $U_i$, $i=0,1,$ where it is homeomorphic to the product $U_i\times [0,\epsilon)$, $\epsilon>0$. Orientations of facets of $F$ and local orders of thin facets along the seams of $F$ restrict to orientations of edges of its boundary 1-foams and local orders of thin edges at vertices of boundary 1-foams using the standard convention for induced orientation of the boundary of a manifold. 

\vspsm

For completeness, we mention that an \emph{oriented $0$-foam} is a finite collection of points with orientations  (signs $+$ and $-$).  
It is clear how to define oriented $1$-foams with boundary.


\subsection{Weighted or \texorpdfstring{$\R_{>0}$}{Rgreaterthan0}-decorated foams}\label{subsec_weighted}

Consider oriented one-foams and two-foams with edges (for 1-foams) and facets (for 2-foams) decorated by real numbers $a$  for various $a>0$ and refer to $a$ as the \emph{thickness}, \emph{width}, or \emph{label} of the facet. At a vertex of a one-foam and a seam of a two-foam widths must add as shown in Figure~\ref{fig6_001}. Informally, one can ``thicken'' the foams  and think of  intervals $[0,a)$ and $[0,b)$ merging into the interval $[0,a+b)=[0,a)\sqcup [a,a+b)$ at a vertex of a 1-foam and a seam of a 2-foam. The order of thin edges near a vertex (for 1-foams) and order of thin 2-facets near a seam (for 2-foams) matches the order of intervals in the merge, see Figure~\ref{fig6_001}. 

\input{fig6_001}

\vspace{0.07in} 

At a vertex of a decorated 2-foam, three thin facets of thickness $a_1,a_2,a_3$  merge into facets of thickness $a_1+a_2$ and $a_2+a_3$, which then merge with the remaining thin facet into the facet of thickness $a_1+a_2+a_3$, see Figure~\ref{fig6_002}, which also shows three parallel cross-sections of this foam.  
\input{fig6_002}

\begin{remark}
If desired, one may allow lines and facets to carry the empty interval $[0,0)$, but this does not seem essential. Such lines and facets can then be deleted from a foam. 
\end{remark}

We call such foams \emph{weighted foams} or \emph{$\R_{>0}$-decorated foams} or \emph{$\IET$-foams} (see later). The definition is straightforward to extend to all dimensions. 
A weighted $0$-foam is a finite set of points with signs $\{+,-\}$ and weights $a>0$.  

Figure~\ref{fig3_019} shows the link of a vertex of a weighted 2-foam. Weighted 2-foams are analogous to measured branched surfaces and measured laminations~\cite{Oer84,Oer88}, but without an embedding into a 3-manifold. 

For $n=0,1$ denote by $\Cob^n_{\indexww}$ the cobordism group of weighted oriented $n$-foams. An $n$-foam $U$ defines the trivial element $[U]=0\in {\Cob^n_{\indexww}}$ iff it bounds a weighted oriented $(n+1)$-foam. 

\input{fig3_019}

\begin{prop}
    The cobordism group of weighted oriented 0-foams is isomorphic to $\R$: 
    \begin{equation}\label{R-iso}
        \CoboRz \ \cong \R. 
    \end{equation}
\end{prop}

\begin{proof}
    A weighted oriented 0-foam is given by a finite collection of points decorated by signs and weights $a>0$. 
    Merge all $+$-decorated points into one point (adding the weights) and all $-$-decorated points into a point (adding the weights). The result is at most two points $(+,a), (-,b)$, which are cobordant to $(+,a-b)$ if $a>b$, $(-,b-a)$ if $a<b$, and to the empty $0$-foam if $a=b$, see Figure~\ref{fig9_007}.

    \vspsm 

\input{fig9_007}

    \vspsm 

     Under the isomorphism in \eqref{R-iso} point $(+,a)$, $a\in \R_{>0}$ is sent to $a\in \R$,  point $(-,a)$ is sent to $-a\in \R$, and the disjoint union of signed decorated points is converted to the sum of corresponding numbers.
\end{proof}


\subsection{Interval exchange transformations and \texorpdfstring{$\R_{>0}$}{R>0}-decorated one-foams} \label{subsec_IET_foams}

Pick $r\ge 1$, a decomposition $1=\sum_{i=1}^r \lambda_i,$ $0<\lambda_i<1, \lambda_i\in \R$ and a permutation $\sigma\in S_r$. Interval exchange transformation $T_{\lambda,\sigma}: [0,1)\lra [0,1)$ is a bijection of a semiclosed interval to itself given by writing it as the disjoint union of $r$ intervals 
\[
[0,1) = [0,\lambda_1)\sqcup [\lambda_1,\lambda_1+\lambda_2)\sqcup \ldots \sqcup [1-\lambda_r,1)  
\]
and permuting the order of intervals according to $\sigma$, making the $i$-th interval $\sigma(i)$-th in the order. 

The Sah-Arnoux--Fathi invariant of $T_{\lambda,\sigma}$ is an element of $\R\otimes_{\Q}\R$ given by 
\begin{equation}\label{eq_saf}   
\SAF(T_{\lambda,\sigma}) := 
\sum_{i=1}^r \lambda_i\otimes t_i = \sum_i\left(\sum_{j:\sigma(j)<\sigma(i)} \lambda_i\otimes \lambda_j - \sum_{j<i}\lambda_i\otimes \lambda_j\right) , 
\end{equation}
where $t_i=\sum_{j:\sigma(j)<\sigma(i)} \lambda_j - \sum_{j<i}\lambda_j\in \R$ is the displacement of the $i$-th interval by $\sigma$. 

One can write $\SAF(T_{\lambda,\sigma})$ as a linear combination of elements $\lambda_i\otimes \lambda_j-\lambda_j\otimes \lambda_i$, $i,j\le r$ and view it as an element of $\R\wedge_{\Q}\R= \Lambda^2_{\Q}(\R)$, the quotient of $\R\otimes_{\Q}\R$ by the abelian subgroup spanned by $\lambda\otimes \lambda, \lambda\in \R$. Note that $\R\otimes_{\Q}\R\cong \Lambda^2_{\Q}(\R) \oplus S^2_{\Q}(\R)$, the sum of symmetric and exterior squares, and one is taking the projection onto the first summand. The invariant can also be written as follows: 
\begin{equation}
    \SAF(T_{\lambda,\sigma}) \ = \ 2 \sum_{i<j:\sigma(j)<\sigma(i)} \lambda_i \wedge \lambda_j, 
\end{equation}
where $a\wedge b$ denotes the image of $a\otimes b$ under the quotient map $q:\R\otimes_{\Q}\R\lra \Lambda^2_{\Q}\R$, since $q(a\otimes b-b\otimes a)=2 a\wedge b$.   

Let $\Aut_{\IET}$ be the group of Interval Exchange Transformations of $[0,1)$, that is, the group of bijections $T_{\lambda,\sigma}$ as above, with the group operation given by the composition of maps.  
There is a short exact sequence of groups 
\begin{equation}\label{ses_iet}
1 \lra [\Aut_{\IET},\Aut_{\IET}]\lra \Aut_{\IET}\stackrel{\SAF}{\lra} \Lambda^2_{\Q}\R \lra 1 .
\end{equation}

\begin{remark}
I.~Zakharevich~\cite{Za2} interpreted the Sah--Arnoux--Fathi invariant as describing $K_1$ of an appropriate assembler category. Combining this result with constructions of the present paper yields an example of the relation between $K_1$ group of an appropriate category and the group of 1-foam cobordisms, in a rather special case. 
In a forthcoming paper we will discuss the relation between the $K_1$ group and cobordism group of decorated 1-foams in greater generality.  
\end{remark}

\begin{remark}
To each interval exchange transformation  $T_{\lambda, \sigma}$ as earlier, we assign a weighted one-foam with boundary $F_{\lambda,\sigma}$ and a closed weighted one-foam $\widehat{F}_{\lambda,\sigma}$, as shown in Figure~\ref{fig6_010}. Start with a line of thickness $1$ and split it into lines of thickness $\lambda_1,\dots, \lambda_r$ from left to right. Then permute the points at the top end of the split via $\sigma$. After that, merge the resulting points into an interval of width 1, and close up top and bottom endpoints, both of thickness 1, into a closed diagram. Denote by $F_{\lambda,\sigma}$ the resulting weighted oriented one-foam with boundary and by $\widehat{F}_{\lambda,\sigma}$ its closure. Intersections in Figure~\ref{fig6_010} are virtual, that is, due to having to depict foam via a projection to the plane. 

\vspace{0.07in}

\input{fig6_010}

Notice that, in the cobordism group $\CoboR$, one-foam $\widehat{F}_{\lambda,\sigma}$ does not depend on the sequence in which the interval 1 is split into $\lambda_1,\dots, \lambda_r$ as long as in the split $\lambda_1,\dots, \lambda_r$ go from left to right. For instance, for $r=3$, the two sequences of splits $1\to (\lambda_1,\lambda_2+\lambda_3)\to (\lambda_1,\lambda_2,\lambda_3)$ and 
$1\to (\lambda_1+\lambda_2,\lambda_3)\to (\lambda_1,\lambda_2,\lambda_3)$ give rise to cobordant foams. 
Likewise, the sequence of merging the intervals back is irrelevant, as long as the order from left to right is $\lambda_{\sigma^{-1}(1)},\dots, \lambda_{\sigma^{-1}(r)}$. The two one-foams that differ in that way are then cobordant via a composition of 2-foams that create vertices, see Figure~\ref{fig6_002}. 
Likewise, foam $F_{\lambda,\sigma}$, in the cobordism set of one-foams with a fixed boundary, does not depend on the order of merges and splits. 
\end{remark}

Composition of two $\IET$s $T_{\lambda,\sigma}$ and $T_{\lambda',\sigma'}$ is an $\IET$ $T_{\lambda'',\sigma''}=T_{\lambda',\sigma'}\circ T_{\lambda,\sigma}$ for suitable $(\lambda'',\sigma'')$. 

\begin{prop}\label{prop_hom}
Foams 
$\widehat{F}_{\lambda'',\sigma''}$ and $\widehat{F}_{\lambda,\sigma}\sqcup \widehat{F}_{\lambda',\sigma'}$ are cobordant. Assigning one-foam $\widehat{F}_{\lambda,\sigma}$ to an $\IET$ $T_{\lambda,\sigma}$ extends to a homomorphism of groups 
\begin{equation}
\phi' \ : \ \Aut_{\IET}\lra \CoboR.
\end{equation} 
\end{prop} 
This is proved by merging $\widehat{F}_{\lambda,\sigma}\sqcup \widehat{F}_{\lambda',\sigma'}$ into a connected foam and converting it to $\widehat{F}_{\lambda'',\sigma''}$ via foam concordance between braid-like foams, see Figure~\ref{fig6_011}. The rules for computing the composition $T_{\lambda',\sigma'}\circ T_{\lambda,\sigma}$ are easy to translate to a composition of concordances between these two 1-foams. 

\input{fig6_011}

$\square$

\subsection{Cobordism group of weighted oriented one-foams}\label{subsec_cob_weighted}

Since the cobordism group is abelian, homomorphism $\phi'$ factors modulo the commutator of the automorphism group, giving a homomorphism
\begin{equation}\label{phi_cob}
    \phi \ : \ \rmH_1(\Aut_{\IET},\Z)\lra \CoboR .
\end{equation}
Figure~\ref{fig7_003} shows a cobordism from a commutator of two elements to the identity (or to the empty 1-foam). 

\vspace{0.1in}

\input{fig7_003}

\begin{theorem}\label{thm_saf}
    Homomorphism $\phi$ in \eqref{phi_cob} is an isomorphism of abelian groups, giving isomorphisms 
    \begin{equation}
        \CoboR \ \cong \ \R\wedge_{\Q}\R\ \cong \ \rmH_1(\Aut_{\IET},\Z)\   \cong \ K_1(\mcC_Z). 
    \end{equation}
\end{theorem}
The second isomorphism is the $\SAF$ invariant, and an isomorphism $\rmH_1(\Aut_{\IET},\Z)\cong K_1(\mcC_Z)$ is constructed in~\cite{Za2}. Category $\mcC_Z$ is the Zacharevich assembler category~\cite{Za2} for the $\IET$s, also see Remark~\ref{remark_Zakharevich} below.

\begin{proof}
    We establish an isomorphism 
    \begin{equation}\label{eq_map_psi}
    \psi \ : \  \CoboR \stackrel{\cong}{\lra}  \R\wedge_{\Q}\R 
    \end{equation}
    which is compatible with  homomorphism $\phi$ and makes the following diagram commute

\begin{center}
$\xymatrix{
\CoboR \ar[rr]^{\psi} & & \R\wedge_{\Q}\R \\ 
& &  \rmH_1(\Aut_{\IET},\Z). \ar[u]_{\frac{1}{2} \cdot \SAF} \ar[llu]^{\phi} \\
}$
\end{center}

Consider a weighted oriented 1-foam $U$ and project it generically to a plane to a diagram $D$. 

The projection has two types of merge points and two types of split points, depending on whether the order of thin edges at a point is clockwise or counterclockwise, see Table~\ref{fig6_004}.  

\vspace{0.1in}

\input{fig6_004}

\vspace{0.1in} 

To diagram $D$  assign an element $\nu(D)\in \R\wedge_{\Q} \R$ as a sum over local contributions: 
\begin{itemize}
    \item A split vertex with a clockwise thin edge order and a merge vertex with a counterclockwise thin edge order contribute $0$,
    \item For the other orientations the contributions are shown in the table in Figure~\ref{fig6_004}.
    \item A crossing of intervals of lengths $a$ and $b$ contributes $a\wedge b$, with orientations of intervals determining the order of $a,b$ in the product, see the table in Figure~\ref{fig6_004}. 
\end{itemize}
Some examples are shown in Figure~\ref{fig6_003}. Note that $\nu$ is additive under the disjoint union of diagrams. 

\input{fig6_003}
 
\vspace{0.07in}

We claim that $\nu(D)$ depends only on the 1-foam $U$, that is, different plane projections result in the same $\nu(D)$.  
It is easy to see that two such projections differ by moves in Figure~\ref{fig6_005} and versions of these moves given by reversing orientation of one or more of the components or reversing the order of thin edges at a vertex. 

\input{fig6_005}

\vspace{0.07in}

It is straightforward to check the invariance of $\nu$ under all variations of moves in Figure~\ref{fig6_005}. 
For example, independent of orientations of $a$ and $b$ lines, invariance of $\nu$ under move 1 in Figure~\ref{fig6_005} is the relation $a\wedge b+b\wedge a=0$. For move 2 it is $a\wedge a =0$. Move 4 and its version for the opposite orientation determine two entries of the Figure~\ref{fig6_004} table given the other three (they determine, for instance, entries 1 and 3 of row 2 given values at entries 2, 4, 5). Move 5 corresponds to the bilinearity property of the tensor product.  

\vspace{0.07in}

Suppose that weighted 1-foams $U_0,U_1$ are cobordant. A cobordism between them can be realized as a finite sequence of elementary cobordisms shown in Figure~\ref{fig6_008}. Note also that the cobordism in the bottom left of the figure is homeomorphic to the identity cobordism, with no topology change between its top and bottom boundaries, which are different projections onto the plane of the same decorated 1-foam. This cobordism is included for completeness, and it corresponds to move 4 in Figure~\ref{fig6_005}.  

\input{fig6_008}

\vspace{0.07in}

For each elementary cobordism
one can pick diagrams for the two 1-foams at its boundaries  so that they differ as shown in Figure~\ref{fig6_008}.
A direct computation implies that, in each case, the two diagrams have the same invariant $\nu$. 

Consequently, the homomorphism $\psi$ in \eqref{eq_map_psi} is well-defined. It is clearly surjective, since each generator $a\wedge b$ is the image of some foam.  Define a homomorphism 
\begin{equation}
    \label{eq_map_psi_prime_two}
    \psi_2 \ : \ \R\otimes_{\Z}\R  \stackrel{\cong}{\lra}  \CoboR 
\end{equation}
by taking $a\otimes b$ to the foam in Figure~\ref{fig6_003} on the right. 

To show that $\psi_2$ is well-defined we need to check the relations  
\begin{eqnarray}
\label{eq_bilin1}
    \psi_2((a_1+a_2)\otimes b) & = & \psi_2(a_1 \otimes b) + \psi_2(a_2 \otimes b),  \\
\label{eq_bilin2}
    \psi_2(a\otimes (b_1+b_2)) & = & \psi_2(a \otimes b_1) + \psi_2(a \otimes b_2),
\end{eqnarray}
which also imply $n\,\psi_2(a\otimes b)=\psi_2(na\otimes b)=\psi_2(a\otimes nb)$, and, since $\Q$ is a divisible group, imply $\psi_2(\frac{a}{n}\otimes b)=\psi_2(a\otimes \frac{b}{n}).$

\vspace{0.07in} 

\input{fig6_009}

The 1-foam $U_{a_1+a_2,b}$ associated to $(a_1+a_2)\otimes b$ is shown in Figure~\ref{fig6_009} on the left. It is cobordant to the foam $U$ with two crossings shown in the middle of the same figure. A crossing can be split off from any foam, as shown in Figure~\ref{fig6_012}. 

\input{fig6_012}

\vspace{0.07in} 

Splitting off both crossings from $U$ results in the foam $U_1$ shown on the right of Figure~\ref{fig6_009}. Foam $U_1$ is the union of $U_{a_1,b},U_{a_2,b}$ and a braid-like foam $U_2$ with no crossings and compatible thin edge orientations at vertices. Foam $U_2$ is cobordant to the circle of weight $a_1+a_2+b$ and, then, to the empty foam. Hence, foams $U_{a_1+a_2,b}$ and $U_{a_1,b}\sqcup U_{a_2,b}$ are cobordant and relation \eqref{eq_bilin1} holds. Relation \eqref{eq_bilin2} follows in the same way. 

\vspace{0.07in}

To check that $\psi_2$ factors through a homomorphism 
\begin{equation}
    \label{eq_map_psi_prime}
    \psi_1 \ : \ \R\wedge_{\Q}\R  \stackrel{\cong}{\lra}  \CoboR 
\end{equation}
we observe that 
\begin{equation*}
    \psi_2(a\otimes b+ b\otimes a) = 0 
\end{equation*}
since the disjoint union $U_{a,b}\sqcup U_{b,a}$ of 1-foams associated to $a\otimes b$ and $b\otimes a$ is null-cobordant. 

To see that $\psi_1$ is surjective, pick a 1-foam $U$. This foam can be represented as the closure of a braid-like 1-foam $B$. Choose a diagram $D$ of $B$ where all splits and merges have local $\nu$-invariant $0$, see Table~\ref{fig6_004}, with the closure $\widehat{D}$ describing the foam $U$.  

All crossings can be removed from $D$ via cobordisms shown in Figure~\ref{fig6_012}. There, as a first step, parallel lines of thickness $a$ and $b$ above and below the crossing are merged to create two intervals, each of thickness $a+b$. They are then brought near each other and merged via a saddle point cobordism. This results in a disconnected 1-foam which is the disjoint union of foam $U_{a,b}$ and a foam with one fewer crossing versus the original. 

One-foam $\widehat{D}$ is cobordant to the union $\widehat{D_1}\sqcup D_2$. Here $\widehat{D_1}$ is the closure of a braid-like crossingless diagram $D_1$ where all merges and splits have local $\nu$-invariant $0$, and $D_2$ is the union of foams $U_{a_i,b_i}$, $a_i,b_i\in \R_{>0}$ over all crossings $(a_i,b_i)$ of $D$.  
Diagram $D_1$ is cobordant to a circle of some thickness and, hence, null-cobordant.  
This shown surjectivity of $\psi_1$.   

Composition $\psi\circ \psi_1$ is clearly identity. That and surjectivity of $\psi_1$ implies that $\psi_1\circ\psi$ is the identity map. 

\end{proof}

\begin{remark}\label{remark_Zakharevich}
On the category side, we can follow Zakharevich~\cite{Za1,Za2} and  consider 
 the category $\mcC_Z$ with objects -- half-open interval $[a,b)\subset \R$. Morphisms are metric-preserving and order-preserving inclusions of intervals, and the  \emph{assembler} structure is given by pairs of morphisms $[a_1,b_1),[a_2,b_2)\stackrel{\psi_1,\psi_2}{\lra} [a,b)$ that cover the interval without overlaps. More generally, given a Zacharevich assembler category $\mcC$, one can consider $n$-dimensional foams where facets are decorated by objects of $\mcC$, $(n-1)$-dimensional seams by coverings of $\mcC$, and so on. The cobordism group of $\mcC$-decorated $n$-foams should then be related to $K_n(\mcC)$ as defined in~\cite{Za1}. 
 \end{remark} 

\begin{remark}
    It is possible to loosely compare the group $\Aut_{\IET}$ of $\IET$ transformations of the interval to the braid group and weighted 1-foams to links (note, though, that 1-foams are not embedded anywhere, while links are embedded in $\R^3$). Closure of a braid is an oriented link and closure of an IET can be described by an oriented weighted 1-foam. The analogue of the Alexander theorem is very simple: any oriented weighted 1-foam is the closure of some element of $\Aut_{\IET}$, and the analogue of the Markov theorem is straightforward to write down as well since 1-foams are not embedded in $\R^3$ (Markov's theorem is known in the harder case of graphs embedded in $\R^3$, see \cite{Is,KT10,CCD}). The analogue of the $\SAF$ invariant for oriented links is, perhaps, the sum of linking numbers $\mathsf{lk}(L_i,L_j)$, $i<j$, over all pairs of components of a link $L$. This analogy is inspired by Figure~\ref{fig6_004}, where $(a,b)$ crossing adds $a\wedge b$ to the $\SAF$ invariant, similar to the formula for the linking number. The $\SAF$ invariant is preserved by cobordisms of oriented weighted 1-foams, as Theorem~\ref{thm_saf} shows. 
    Linking number is invariant under some cobordisms in $\R^3\times [0,1]$ between links in $\R^3$. More precisely, pick an ordered countable set $S$ and equip a link $L$ with a map $\psi:\mathrm{comp}(L)\lra S$ from its set of connected components to $S$. Consider cobordisms $M$ between such links $L,L'$ equipped with a map $\mathrm{comp}(M)\lra S$ which is compatible with the maps $\psi,\psi'$ for its boundary links $L,L'$. The $S$-linking number
    \[
    \mathsf{lk}_S(L) \ := \ \sum_{i,j|\psi(i)<\psi(j)} \mathsf{lk}(L_i,L_j)
    \]
    is invariant under such cobordisms.  
\end{remark}

\begin{remark}\label{remark_H}
    In the definition of weighted foams abelian semigroup $(\R_{>0},+)$ can be replaced by an arbitrary commutative semigroup $(H,+)$. One can then form the abelian group $\Cob^1_H$ of $H$-weighted oriented 1-foams modulo cobordisms. The latter are $H$-weighted oriented 2-foams with boundary. The above arguments extend to an isomorphism 
    \begin{equation}
        H\wedge H \ \cong \  \Cob^1_H 
    \end{equation}
    taking $a\wedge b$ to $[U_{a,b}]$. Here $H\wedge H$ is the abelian group generated by symbols $a\wedge b$, $a,b\in H$ with defining relations
    \begin{eqnarray*}
        a\wedge b + b\wedge a & = & 0, \\
        (a_1+a_2)\wedge b & = & a_1\wedge b + a_2\wedge b. 
    \end{eqnarray*}
    In particular, the cobordism group of $\R$-decorated oriented 1-foams is isomorphic to that of $\R_{>0}$-decorated foams, since the natural map $\R_{>0}\wedge \R_{>0}\lra \R\wedge \R$ induced by the inclusion $\R_{>0}\hookrightarrow \R$ is an isomorphism. 
\end{remark}

There are several related ways to thicken an (oriented) $\R_{>0}$-weighted one-foam to a two-dimensional structure and a cobordism between such foams to a three-dimensional structure. 

\vspace{0.07in} 

{\it I. Lower limit topology.}  One can thicken an $\R_{>0}$-decorated 1-foam to a 2-dimensional structure by multiplying a 1-facet $I$ carrying label $a$ by $[0,a)$ and then gluing these products at vertices, see Figures~\ref{fig6_006}, \ref{fig6_007}.

\input{fig6_006}

\input{fig6_007}

We equip intervals $[0,a)$ with the \emph{lower limit topology $\ell$}, with a basis of open sets given by $[a_1,b_1)$, with $0\le a_1 < b_1 \le a$, see Munkres~\cite[Section 13]{Mun00}. With this topology, there are homeomorphisms $[0,a)\sqcup [0,b)\cong [0,a+b)$ given by placing $[0,b)$ immediately to the right of $[0,a)$. 

In this way, a one-foam $U$ as above is thickened to a topological space $T(U)$ which is locally homeomorphic to the product $[0,1)_{\ell}\times (0,1)$. A cobordism between two such one-foams is thickened to a topological space locally homeomorphic to $[0,1)_{\ell}\times (0,1)^2$.

This thickening of one-foams and two-foams is related to the Zakharevich assembler category, see Remark~\ref{remark_Zakharevich} above and the discussion in the introduction. 

Topological space $T(U)$ associated to a 1-foam $U$ carries a foliation where connected components of the leaves are locally $x\times (0,1)$ for $0\le x < a$. If all leaves are compact (and then necessarily homeomorphic to $\mathbb{S}^1$), the foam $U$ is null-cobordant. The opposite implication fails, since $U\sqcup U^!$ is null-cobordant for any $U$. 

\vspsm 

\input{fig7_004}

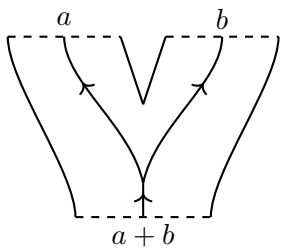
{\it II. Train tracks on surfaces.}  A weighted oriented 1-foam can be thickened to an oriented train track~\cite{PH92} on a surface with boundary, see Figure~\ref{fig7_004}. Transformations of unoriented train tracks that do not change the associated measured foliation or measured lamination~\cite[Sections 2.1, 2.3]{PH92} can be interpreted as cobordisms of train tracks in $S\times [0,1]$, where $S$ is the surface that contains the train track. 

\begin{remark}
Interval exchange transformations can be thickened to very flat surfaces, or translation surfaces~\cite{Zor_NT_Phys_Geom06,Zor_ICM06}, i.e., surfaces with a flat metric and singular points where total angles at these points are multiples of $2 \pi$. Oriented weighted one-foams, equipped with additional data, can likewise be thickened to very flat surfaces (we omit the details). 
\end{remark}

  
\section{Planar unoriented weighted foams and antisymmetric 2-brackets}  \label{subsec_planar}
    
Consider weighted unoriented 1-foams $U$ embedded in the plane $\R^2$, and denote an embedded foam also by $U$. Such a foam $U$ is analogous to a weighted unoriented train track on a surface~\cite{PH92}, except that no conditions are imposed on the Euler characteristic of components of the complement of $U$ in $\R^2$ (compare with~\cite[Section 1.1]{PH92}). An embedded 1-foam can be thickened to an open subset of $\R^2$ with an unoriented bidirectional flow on it, see Figure~\ref{fig9_001}.

\vspsm

\input{fig9_001}

\vspsm

By a cobordism between two unoriented embedded 1-foams $U_0,U_1$ we mean an unoriented embedded 2-foam $V\subset \R^2\times [0,1]$ so that $V\cap (\R^2\times \{i\})=U_i$, $i=0,1$. Note that for any 1-foam $U$ the disjoint union $U\sqcup U^!$ of $U$ with its mirror image is null-cobordant. See Figure~\ref{fig9_002} for an example of the mirror image of an unoriented embedded foam. 

\vspsm

\input{fig9_002}

\vspsm

Denote by $\Cob^{1,\up}_{\R_{>0}}$ the set of cobordism classes of unoriented embedded one-foams (``up'' in the superscript stands for \emph{unoriented planar}). The disjoint union and mirror image operations turn this set into an abelian group. Denote by $[U]$ the image of a 1-foam $U$ in that group. 

In general, there is no obvious cobordism between $U$ and $U^!$ (and we will see that $[U]\not= [U^!]$, in general). 

For $a,b> 0$ denote by $T(a,b)$  the foam shown in Figure~\ref{fig9_002}, which we also call a \emph{tripod foam}. 
Note that $T(a,b)^!\equiv  T(b,a)$.

\begin{prop} \label{prop_generated} The group $\Cob^{1,\up}_{\R_{>0}}$ is generated by symbols $[T(a,b)]$ of tripod 1-foams over all $a,b>0$. 
\end{prop} 

\input{fig9_003}

\input{fig9_004}

\begin{proof}
     The cobordism shown in Figures~\ref{fig9_003},~\ref{fig9_004} allows to convert an interval into two looped half-invervals. The loop at the end of an $a$-interval has thickness $a/2$. This cobordism can be applied at each edge of $U$, as shown in Figure~\ref{fig9_013} on the left, to cut $U$ into a union of tripod foams and circles. Each circle can further be cut into a barbell foam (the latter is shown in Figure~\ref{fig9_014}, together with a cobordism from it to the empty foam, in the top right corner of the figure). 
If a  foam $U$ has vertices $v_1, \dots, v_n$ with thin edges at the vertex $v_i$ of thickness $(a_i,b_i)$, going counterclockwise, then $[U]=\sum_{i=1}^n [T(a_i,b_i)].$ 
\end{proof}

\input{fig9_013}

\input{fig9_014}

\begin{remark} Figure~\ref{fig9_005} shows that $T(a,b)$ is cobordant to $T(a,b-a)$ if $a<b$. Passing to mirror images shows that $T(a,b)$ is cobordant to $T(a-b,b)$ if $a>b$. These cobordisms can be iterated to a foam cobordism version of the Euclidean division algorithm. In particular, iterating these operations we see that $T(a,b)$ is null-cobordant if $b\in \Q a$ (that is, if $a$ and $b$ are proportional over $\Q$). Cobordism between $T(a,a)$ and the empty foam is shown in Figure~\ref{fig9_014}.   
\end{remark}

\input{fig9_005}

\vspsm

Consider 1-foams in the first two rows of Figure~\ref{fig6_008}, ignoring orientations of edges and orders of thin edges at vertices and instead viewing the 1-foams as planar (embedded in $\R^2$). These 1-foams are cobordant in pairs, via 2-foam cobordisms embedded in $\R^2\times [0,1]$. At the same time, breaking up these 1-foams along edges results in disjoint unions of foam $T(x,y)$ for various $x,y\in \R_{>0}$. Passing to the cobordism group and replacing $\R_{>0}$ by a commutative semigroup $H$ motivates the following definition.   

\vspsm 

Given a commutative semigroup $(H,+)$, denote by $\Ztwo(H)$ the abelian group with generators $[a,b]$, $a,b\in H$, and defining relations 
\begin{eqnarray}
   \label{eq_rel_a} [a,a] & = & 0, \ \  a\in H, \\
   \label{eq_rel_b}
    [a,b]+[b,a] & = & 0, \ \ a,b\in H, \\
    \label{eq_rel_c} [a,b]+[a+b,c] & = & [a,b+c]+[b,c], \ \  a,b,c\in H. 
\end{eqnarray}
Note that relation \eqref{eq_rel_a} does not imply skew-commutativity relation \eqref{eq_rel_b} since the bracket $[a,b]$ is not bilinear. Equations \eqref{eq_rel_a} and \eqref{eq_rel_b} together are the strong version of the skewcommutativity property in the absence of bilinearity. Equation \eqref{eq_rel_c} is reminiscent of the 2-cocycle relation -- the difference between the two sides can be interpreted as the signed boundary of a 3-simplex with oriented edges labelled $a,b,c,a+b,b+c,a+b+c$.
This is explained in Figure~\ref{fig9_020}. The analogue of symbol $[x,y]$ is an oriented triangle with oriented sides labelled $x,y,x+y$. The oriented boundary of a 3-simplex with sides labelled by $a,b,c$ and their sums is the difference between the RHS and LHS of equation \eqref{eq_rel_c}. 

We call $\Ztwo(H)$ \emph{the antisymmetric 2-bracket} or \emph{antisymmetric 2-cocycle of $H$}. 

\input{fig9_020}

A homomorphism $f:H_1\lra H_2$ of commutative semigroups induces a homomorphism $\Ztwo(H_1)\lra \Ztwo(H_2)$.

\begin{prop} \label{prop_two_groups}
The cobordism group $\Cob^{1,\up}_{\R_{>0}}$ of planar unoriented weighted 1-foams is isomorphic to  $\Ztwo(\R_{>0})$: 
\begin{equation}
    \Cob^{1,\up}_{\R_{>0}} \ \cong \ \Ztwo(\R_{>0}) 
\end{equation}
taking $[T(a,b)]$ to $[a,b]$ for all $a,b>0$. 
\end{prop} 
\begin{proof} 
Consider the free abelian group $Z$ on generators $[a,b]'$, over all $a,b\in \R_{>0}$. Proposition~\ref{prop_generated} says that there is a surjective homomorphism  $\tau: Z\lra \Cob^{1,\up}_{\R_{>0}}$ taking $[a,b]'$ to $[T(a,b)]$. Furthermore, relations \eqref{eq_rel_a}-\eqref{eq_rel_c} hold for the images of $[a,b]'$ under $\tau$. Indeed, $T(a,a)$ is null-cobordant, giving the relation $\tau([a,a]')=0$. The disjoint union $T(a,b)\sqcup T(b,a)$ is null-cobordant, implying 
\begin{equation}\label{eq_tau_prime_1}
\tau([a,b]'+[b,a]')=0.
\end{equation} 
It is convenient to pair up $a$- and $b$-lollipop ends of $T(a,b)\sqcup T(b,a)$ and pass to the one-foam which is a split of $(a+b)$-strand into $a$- and $b$-strands, followed by the merge, see Figure~\ref{fig9_014} top left. There is a natural cobordism from the split-merge to the $(a+b)$-strand, which is another way to see the relation \eqref{eq_tau_prime_1}. Ignoring orientations and edge orders, this cobordism is depicted in the top left corner of Figure~\ref{fig6_008}.  Likewise, 
that 
\begin{equation}\label{eq_tau_prime_2}
\tau([a,b]')+\tau([a+b,c]')=\tau([a,b+c]')+\tau([b,c]')
\end{equation}
follows from them existence of  a cobordism between the two ways to merge parallel $a,b,c$-strands into $(a+b+c)$-strand, see Figure~\ref{fig9_013} on the right. For example, there is the one-vertex cobordism between these two 1-foams. 

Consequently, homomorphism $\tau$ descends to a surjective homomorphism, also denoted 
\begin{equation}\label{eq_tau}
\tau\ :\ \Ztwo(\R_{>0})\lra \Cob^{1,\up}_{\R>0}. 
\end{equation} 
Vice versa, breaking a planar weighted one-foam into tripods gives a map $\tau'$ from planar foams into $\Z$-linear combinations of symbols $[a,b]'$, and we would like to turn $\tau'$ into the inverse of $\tau$. A cobordism between two one-foams can be represented as a composition of elementary cobordisms, including vertex cobordisms, singular saddles, cups and caps, and the usual saddle, cup and cap cobordisms between 1-manifolds. These cobordisms do not change the linear combination of symbols $[a,b]$ associated to a one-foam, when viewed as an element of $\Ztwo(\R_{>0})$.

Note that the relation $[a,a]=0$ in \eqref{eq_rel_a} does not come from any elementary cobordism. The tripod $T(a,a)$ is null-cobordant, however, as shown in  Figure~\ref{fig9_014}. This discrepancy has the following explanation. When breaking a tripod $T(a,b)$ along every edge to construct the map $\tau'$  one adds three more terms to $[a,b]'$ due to the three lollipop vertices of the tripod, so that the composition of $\tau$ and $\tau'$ is 
\[ [a,b]'\stackrel{\tau}{\lra} [T(a,b)]\stackrel{\tau'}{\lra} [a,b]'+ [a/2,a/2]'+[b/2,b/2]'+[(a+b)/2,(a+b)/2]'.
\]
In particular, the composition $\tau'\tau$ differs from the identity due to the presence of three terms $[x,x]'$ for $x\in \{a/2,b/2,(a+b)/2\}$. Setting these terms to $0$ in $\Ztwo(\R_{>0})$ makes the composition $\tau'\tau =\id$, where now $\tau'$ is a well-defined map 
\begin{equation}\label{eq_tau_prime}
\tau' \ :\ \Cob^{1,\up}_{\R>0}\lra \Ztwo(\R_{>0}), \ \ [T(a,b)]\stackrel{\tau'}{\lra}[a,b]. 
\end{equation} 
In the other direction, it is clear that $\tau\tau'=\id$. 
Consequently, homomorphism $\tau$ in \eqref{eq_tau} is an isomorphism. 
\end{proof} 

Extending from $\R_{>0}$ to $\R$ and 
adding bilinearity relations on the symbols $[a,b]$, so that, in addition $[a_1+a_2,b]=[a_1,b]+[a_2,b]$, gives a surjective homomorphism 
\begin{equation}
    \theta' \ : \ \Ztwo(\R_{>0}) \lra \R\wedge_{\Z} \R \cong \R\wedge_{\Q} \R , 
\end{equation}
and, consequently, a surjective homomorphism 
\begin{equation}
    \theta \ : \ \Cob^{1,\up}_{\R_{>0}}\lra \R\wedge_{\Q}\R
\end{equation}
taking $[T(a,b)]$ to $a\wedge b$ (compare with the $\SAF$ invariant, see Section~\ref{sec_iet}). This allows to show that some unoriented planar 1-foams are not null-cobordant. 

\begin{corollary}
    Planar unoriented foam $T(a,b)$ for $a,b\in \R_{>0}$ is not null-cobordant if $b\notin \Q a$.
\end{corollary}

It turns out that the bracket $[a,b]$ is almost bilinear, as explained by the following result. 

\begin{prop}\label{prop_kernel}
    The kernels of $\theta'$ and $\theta$ consist of elements of order at most two. For any $a,b_1,b_2\in \R_{>0}$ the following relation  holds in $\Ztwo(\R_{>0})$:  
    \begin{equation}\label{eq_almost_bilin}
        2([a,b_1+b_2]-[a,b_1]-[a,b_2]) \ = \ 0.  
    \end{equation}
\end{prop}

\begin{proof}
Consider the following three equations:
\begin{eqnarray}\label{almost_bilin_1}
    [a,b_1+b_2] + [b_1,b_2] & = & [a+b_1,b_2] + [a,b_1], \\
\label{almost_bilin_2}
[b_1,b_2+a] + [b_2,a] & = & [b_1+b_2,a] + [b_1,b_2],\\
 \label{almost_bilin_3}
    [b_1,a+b_2] + [a,b_2] & = & [a+b_1,b_2]+ [b_1,a]. 
\end{eqnarray}
Equation~\eqref{almost_bilin_1} is the 2-cocycle relation, for $a,b_1,b_2$. Equation~\eqref{almost_bilin_2} is given by cyclicly permuting the terms of the previous equation, $a\mapsto b_1\mapsto b_2\mapsto a$. Equation~\eqref{almost_bilin_3} is given by transposing $a$ and $b_1$ in~\eqref{almost_bilin_1}. 
Writing down the linear combination \eqref{almost_bilin_1}+\eqref{almost_bilin_2}-\eqref{almost_bilin_3} and using that the bracket is antisymmetric gives relation \eqref{eq_almost_bilin}. 

This argument is borrowed from~\cite{Bilin_almost}, which shows bilinearity of the difference $[a,b]-[b,a]$ assuming only the 2-cocycle equation~\eqref{eq_rel_c}  for all $a,b\in H$, where $H$ is an abelian group. When the 2-cocycle is, additionally, antisymmetric, via equation~\eqref{eq_rel_b}, the difference $[a,b]-[b,a]=2[a,b]$.  
\end{proof}
The proposition tells us that the bracket $[a,b]$ is ``almost'' bilinear, with the difference $[a,b_1+b_2]-[a,b_1]-[a,b_2]$ either $0$ or an element of order $2$. 

\begin{corollary}\label{cor_twice_U}
    The foam $U\sqcup U$, where
    \[
    U \ = \ T(a,b_1+b_2)\sqcup T(b_1,a)\sqcup T(b_2,a),
    \]
    is null-cobordant for any $a,b_1,b_2>0$. 
\end{corollary}

Foam $U$ is shown in Figure~\ref{fig9_021}. 

\input{fig9_021}

We do not know whether the scalar $2$ can be dropped from equation \eqref{eq_almost_bilin}, so that $[a,b]$ is bilinear in $a,b$. That would be equivalent to foams $U$ in Corollary~\ref{cor_twice_U} being null-cobordant for all $a,b_1,b_2>0$.

To further study abelian groups in   Proposition~\ref{prop_two_groups} it is natural to extend possible weights of foam facets from positive to all real numbers. First, we discuss the group $\Ztwo(H)$ for general commutative semigroups $H$, having $(\R,+)$ in mind. Note that Proposition~\ref{prop_kernel} holds for any commutative semigroup $H$ in place of $\R_{>0}$, so that there is an exact sequence 
\[
  0 \lra \ker \: \theta' \lra Z^2(H) \stackrel{\theta'}{\lra} H\wedge' H \lra 0
\]
with $2x=0$ for $x\in \ker \: \theta' $. Here $H\wedge' H$ is the abelian group which is the quotient of the abelian group closure of $H\otimes_{\Z} H$ by the relations $a\wedge' b+b\wedge' a=0$ and $a\wedge' a=0$, by analogy with \eqref{eq_rel_a},~\eqref{eq_rel_b}. Symbol $\wedge'$ is used instead of $\wedge$ since the relation $a\wedge a=0$ is usually not imposed in the definition of the exterior square (but follows for 2-divisible semigroups). 

\vspsm 

If $0\in H$, then \eqref{eq_rel_c} with $(a,b,c)=(a,0,b)$ implies that $[a,0]=[0,b]$ for all $a,b\in H$. Specializing to $b=0$ gives 
\begin{equation}\label{eq_zero}
    [a,0] \ = \ [0,a] \ = \ 0, \ \forall a\in H. 
\end{equation}

\begin{prop}\label{prop_rel}
    Assume that  $0\in H$. 
    Then 
    \begin{enumerate}
        \item $[a,0]=0$, $\forall a\in H$, 
        \item If $-a\in H$ (i.e., $a$ is invertible in $H$) then 
        \begin{eqnarray}
            2[a,-a] & = & 0, \\ \label{eq_b_ma} [b,-a] & = & [a,b-a]+[a,-a], \ \forall b\in H,   \\
            \label{eq_two_a} [2a,-2a] & = & 0,
        \end{eqnarray}
        \item If $-a,-b\in H$ then 
        \begin{equation}\label{eq_ma_mb} [-a,-b]=[a,b]+[a,-a]+[b,-b]-[a+b,-a-b].
        \end{equation}
    \end{enumerate}

\end{prop}
\begin{proof} See \eqref{eq_zero} for (1).
   Notice that relation \eqref{eq_rel_c} can be visualized as the ``associativity'' property for merging $a,b,c$ into $a+b+c$ in two possible ways, where a vertex merging $x,y$ contributes $[x,y]$ to the sum, see Figure~\ref{fig9_008}.  

\input{fig9_008} 

\vspsm 

Iterating this associativity relation gives us a relation between any two tree diagrams for merging $(a_1,\ldots,a_n)$ into $a_1+\ldots +a_n$. Now apply the relation to the two trees shown in Figure~\ref{fig9_008} on the right  merging $(a,-a,a,-a)$ to $0$ and use that $[a,-a]+[-a,a]=0$ and $[b,0]=0$ for any $b$ to conclude that $2[a,-a]=0$.

    For the relation \eqref{eq_b_ma}, apply \eqref{eq_rel_c} to $(b-a,a,-a)$ to get $[b-a,a]+[b,-a]=[b-a,0]+[a,-a]$.
    For the relation \eqref{eq_two_a},  two of the ways to merge $(a,a,-a,-a)$ to $0$ give 
    \begin{equation}\label{eq_two_aa}
    [a,a]+[-a,-a]+[2a,-2a] = [a,-a]+[a,0]+[a,-a],
    \end{equation}
    resulting in $[2a,-2a]=2[a,-a]=0$
    
    For the relation \eqref{eq_ma_mb}, 
    apply \eqref{eq_rel_c} to $(a,-a,-b)$ and $(-a-b,a,b)$. 
\end{proof}

Notice that, modulo terms $[x,-x]$, relations \eqref{eq_b_ma} and \eqref{eq_ma_mb} are $[b,-a]\sim [a,b-a]$ and $[-a,-b]\sim [a,b]$.

\begin{remark}
    Let $H=(\Z/4,+)=\{0,1,2,3\}$. It is tedious but straightforward to check that the map 
    \begin{equation}
        \psi([a,b])=\begin{cases} 0 &\mathrm{if} \ a=0 \ \mathrm{or}\ b =0 \ \mathrm{or} \ a=b, \\
        1 &\mathrm{otherwise}
        \end{cases}
    \end{equation}
extends to a homomorphism $\psi:\Ztwo(\Z/4)\lra \Z/2$. Under this homomorphism the image of $[1,-1]=[1,3]$ is nontrivial. Via the surjective homomorphism $\Z\lra \Z/4$ we see that $[1,-1]$ in nontrivial in $\Ztwo(\Z)$ as well. Consequently, $[a,-a]$ is not always $0$ in $\Ztwo(H)$ for $a,-a\in H$. 
\end{remark}

Elements $[a,-a]$, over all $a,-a\in H$, generate a 2-torsion subgroup in $\Ztwo(H)$, which we can denote $\Ztwo_-(H)$. This subgroup is trivial if $H$ is 2-divisible, in view of the relation \eqref{eq_two_a}. In particular, it is trivial for $H=(\R,+)$.

We denote by $\R_{>0}$ the semigroup $(\R_{>0},+)$ and by $\R$ the group $(\R,+)$. 
Semigroup $(\R_{>0},+)$ is not a monoid, that is, $0\notin \R_{>0}$. The inclusion $\R_{>0}\subset \R$ induces a homomorphism  
\begin{equation}
\rho \ : \  \Ztwo(\R_{>0})\lra \Ztwo(\R).
\end{equation}
To differentiate between elements of the two groups denote by $[a,b]_{\R}$ the symbol of the pair $a,b\in \R$ viewed as an element of $\Ztwo(\R)$. The map $\rho$ is given by 
$\rho([a,b])=[a,b]_{\R}$ for $a,b>0$.

\begin{corollary}\label{cor_hold}
    In $\Ztwo(\R)$ and for $a,b>0$, the following relations hold: 
    \begin{equation}\label{eq_bend_a}
        [a,-b]_{\R}   =   \begin{cases} [b,a-b]_{\R} &  \mathrm{if} \ a>b, \\
        [b-a,a]_{\R} & \mathrm{if} \ a< b, \\
        0 & \mathrm{if} \ a= b, 
        \end{cases} 
    \end{equation}
    \begin{equation}
        [-a,b]_{\R} = - [b,-a]_{\R}, \ \  
        [-a,-b]_{\R} = [a,b]_{\R}. 
    \end{equation}
\end{corollary}
\begin{proof}
These relations are obtained by dropping off terms $[x,-x]$ from the relations in Proposition~\ref{prop_rel}. Terms $[x,-x]_{\R}=0$ since $\R$ is 2-divisible. 
\end{proof}

Corollary~\ref{cor_hold} implies that $\rho$ is surjective, since the symbol $[a,b]_{\R}$ with at least one of $a,b$ negative can be written as $\pm \rho([a',b'])$ for suitable $a',b'\in \R_{>0}$, or $[a,b]_{\R}=0$.
\begin{prop}\label{prop_iso}
    Homomorphism 
    \[
    \rho \ :\ \Ztwo(\R_{>0})\lra \Ztwo(\R)
    \]
    induced by the inclusion $\R_{>0}\subset \R$  is an isomorphism. 
\end{prop}

\begin{proof}
    Corollary~\ref{cor_hold} relations can be used to define a map from symbols $[a,b]_{\R}$ with $a,b\in \R$ to signed symbols $[a,b]$ with positive $a,b$. Consider the map $\delta$ defined on symbols as follows and assuming $a,b>0$:  
    \begin{eqnarray}
    \delta([a,b]_{\R}) & = & \delta([-a,-b]_{\R}) \ = \ [a,b], \\ 
    \delta([a,-b]_{\R}) & = & [b,a-b], \ \mathrm{if} \ a>b, \\
    \delta([a,-b]_{\R}) & = & [b-a,a], \ \mathrm{if} \ a<b, \\
    \delta([-a,b]_{\R}) & = & - \delta([b,-a]_{\R}), \\
    \delta([a,-a]_{\R}) & = & 0 . 
    \end{eqnarray}
    We claim that $\delta$ extends to a well-defined homomorphism 
    $\delta:\Ztwo(\R)\lra \Ztwo(\R_{>0})$. This map respects the relations \eqref{eq_rel_a} and \eqref{eq_rel_b}. A tedious case-by-case verification shows that it also respects the relation \eqref{eq_rel_c}. For example, consider relation \eqref{eq_rel_c} for the triple $(a,-b,c)$ where $c>b>a>0$. To check that 
    \[
    \delta([a,-b]_{\R})+\delta([a-b,c]_{\R}) = \delta([a,c-b]_{\R}) + \delta([-b,c]_{\R}), 
    \]
    we compute the two sides: 
    \begin{eqnarray*}
       \LHS & = & [b-a,a]+[a+c-b,b-a] , \\
       \RHS & = & [a,c-b]+[c-b,b],
    \end{eqnarray*}
    and write 
    \begin{eqnarray*}
    [a,c-b]+[c-b,b] & = & ([c-b,b] + [a,b-a])+[a,b-c]-[a,b-a] \\
    & = & ([c-b,a]+[a+c-b,b-a])+[a,b-c]-[a,b-a] \\
    & = & [a+c-b,b-a] + [b-a,a]\  =\  \LHS.
    \end{eqnarray*}
    The case $a>b>c$ follows by symmetry, and other cases to consider are $a>b,c>b$; $b>a+c$; $a+c>b, b>a,b>c$. All of them together take care of the relation \eqref{eq_rel_c} when only the middle number is negative. The case $(-a,-b,-c)$, i.e., all three numbers are negative, is trivial, but there are many other cases. They follow via straightforward computations which are omitted.
\end{proof}

The group $\Ztwo(H)$ depends only on the isomorphism class of abelian semigroup $H$. Thinking of $\R$ an abelian group and using the axiom of choice one can write $\R\cong \oplus_{J} \Q$, where the index set $J$ is uncountable. 
 Consequently, $\Ztwo(\R_{>0})\cong \Ztwo(\R)\cong \Ztwo(\oplus_{J} \Q),$ giving a more symmetric presentation of $\Ztwo(\R_{>0})$. This does not give an explicit  description of $\Ztwo(\R_{>0})$, just a description with more internal symmetries, but in our study of this group we stop here. 
A natural question would be to understand the kernel of the surjective homomorphism $\theta':\Ztwo(\R_{>0})\lra \Lambda^2_{\Q}(\R)$ sending $[a,b]$ to $a\wedge b$. From Proposition~\ref{prop_kernel} we know that $2x=0$ for any element $x\in \ker(\theta')$.

\begin{remark}  Note that $\Ztwo(\Q_{>0})\cong \Ztwo(\Q)=0$. This can be derived from all tripods $T(a,b)$ for $a,b\in \Q_{>0}$ being null-homotopic. A related observation is that thickening $T(a,b)$ with rational $a,b$ results in a foliated planar surface with all leaves closed and diffeomorphic to $\S^1$. 
\end{remark} 

Proposition~\ref{prop_iso} shows that passing from $\R_{>0}$ to $\R$  does not change the group $\Ztwo$.  Let us consider unoriented planar one-foams where edges are labelled by real numbers rather than just positive numbers (planar $\R$-weighted one-foams). A cobordism between two such foams is given by an unoriented $\R$-decorated two-foam in $\R^2\times [0,1]$. 
An $\R$-weighted two-foam also has vertices with local structure as in Figures~\ref{fig6_002} and~\ref{fig9_006}, but now $a,b,c$ are arbitrary real numbers, possibly $0$. 
Denote by $\Cob^{1,\up}_{\R}$ the cobordism group of $\R$-weighted planar unoriented one-foams. There is a natural homomorphism 
\begin{equation}
    \label{eq_iota}
    \iota \ : \ \Cob^{1,\up}_{\R_{>0}} \lra \Cob^{1,\up}_{\R}
\end{equation}
given by viewing $\R_{>0}$-weighted one- and two-foams as $\R$-weighted foams. Likewise, there is a homomorphism 
\begin{equation}\label{eq_tau_R}
\tau_R\ :\ \Ztwo(\R)\lra \Cob^{1,\up}_{\R} 
\end{equation} 
defined analogously to the  homomorphism  \eqref{eq_tau}. Map $\tau_R$ takes the symbol $[a,b]_{\R}$ to the concordance class of the tripod $T(a,b)$, where now weights may be non-positive. 

\begin{theorem}\label{thm_iso_iota}
    Maps $\iota$ and $\tau_R$ are isomorphism of abelian groups. 
\end{theorem}

\begin{proof}
That $\tau_R$ is an isomorphism can be shown in the same way as for $\tau$, see the proof of Proposition~\ref{prop_two_groups}. 
Next, observe 
that formulas in Corollary~\ref{cor_hold} convert symbols $[x,y]_{\R}$ when one of both $x,y$ are negative into symbols with positive entries. We now define the foam counterpart of these formulas. Start with an $\R$-weighted one-foam $U$ and convert it to an $\R_{\ge 0}$-weighted one-foam $U^{\circ}$ as follows. 
First, convert each line $a$ into a line of weight $|a|$, for $a\in \R^{\ast}$, see Figure~\ref{fig9_009}.

\input{fig9_009}

\input{fig9_010}

At vertices of $U$ edges of negative weight are bent to the opposite side to retain the balance of weights at a vertex.  Figure~\ref{fig9_010} shows how a single negative edge is bent at a vertex. 
Figure~\ref{fig9_011} shows modifications at a vertex if two out of three edges have negative weights. In Figure~\ref{fig9_012} we see that an $(a,-a)$ vertex gets smoothed out into part of a segment, and that no bending is necessary at an $(-a,-b)$ vertex, just weight reversal at all three edges of the vertex. 

\input{fig9_011}

\input{fig9_012}

Foam $U^{\circ}$ may have edges (and circles) of weight $0$. A circle of any weight is null-cobordant even if there is a 1-foam inside the disk that it bounds, by converting the circle to a barbell. 
Given a $0$-edge $e$, applying Figure~\ref{fig9_013} (left) transformation at the two endpoints of $e$ produces tripod foams $T(a_1,0)$ and $T(a_2,0)$ (or their reflections) for some $a_1,a_2$. 
These foams are null-cobordant (see Figure~\ref{fig9_015}), and cobordant to barbell foams with weights $a_1,a_2$ (the latter are null-cobordant as well, see Figure~\ref{fig9_014}). Inserting these barbell foams back into the original 1-foam and composing these cobordisms shows that an $\R_{\ge 0}$-weighted 1-foam $V$ with a $0$-weight edge $e$ is cobordant to the same foam with edge $e$ deleted. Thus, all edges and circles of weight $0$ (components of weight $0$) can be deleted from an  $\R_{\ge 0}$-weighted foam $V$, resulting in a cobordant $\R_{>0}$-foam. In particular, this is shown as the second step in the top row of Figure~\ref{fig9_012}. 

\vspsm 

\input{fig9_015}

Denote by $U^{\bullet}$ the foam $U^{\circ}$ with weight $0$ components removed. 
The map $U\mapsto U^{\bullet}$ from planar $\R$-weighted 1-foams to planar $\R_{>0}$-weighted 1-foams needs to be extended to cobordisms between 1-foams, that is, to 2-foams with boundary.

Suppose that $F$ is an $\R$-weighted two-foam with boundary $U$, unoriented and embedded in $\R^2\times [0,1)$, with $U\cong \partial F \cong F\cap (\R^2\times \{0\})$. We convert all facets of $F$ with negative labels $-a$ to positive labels $a>0$. 

At each seam of $F$ two facets merge into one. If one  or two of these facets had negative weights, we make these facets approach the seams from the opposite side, by taking the rules in Figures~\ref{fig9_010},~\ref{fig9_011},~\ref{fig9_012} and multiplying them by the interval to get the corresponding rules for 2-foams. These modifications are  depicted in Figure~\ref{fig9_016}.  

\input{fig9_016}

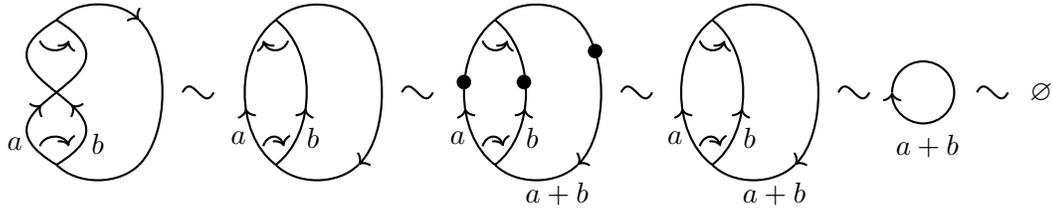
Next, one produces modification rules at vertices of $F$, where facets have weights $a,b,c,a+b,b+c,a+b+c$, for some $a,b,c\in \R$. Taking the link of a vertex results in a 2-foam $L(a,b,c)$ on the 2-sphere (see Figure~\ref{fig3_019} on the left, with orientations and thin edge orders at nodes dropped). Converting it to $L(a,b,c)^{\bullet}$, one needs to check that it is null-cobordant through a $\R_{>0}$-weighted foam and pick a particular cobordism to replace each $(a,b,c)$-vertex of an $\R$-weighted 2-foam. This is done on a case-by-case basis, and the rest of the proof closely resembles that of Proposition~\ref{prop_iso} towards the end. Here we provide the cobordisms in two out of the many cases here. Instead of the cobordism from $L(a,b,c)^{\bullet}$ to the empty foam we depict cobordisms between two possible ways to merge $a,b,c$ edges into the $a+b+c$ edge, see Figure~\ref{fig9_013} on the right. 

We consider the case when the middle number is negative and write it as $-b$. Since intermediate edges are $a-b$ and $c-b$, there are four cases to consider: 
\begin{enumerate}
    \item $a\ge b, c\ge b\ge 0$, 
    \item $a\ge b \ge c\ge 0$ (case $c\ge b\ge a\ge 0$ is given by reflection), 
    \item $b\ge a, b\ge c, a+c\ge b$, $a,c\ge 0$,  
    \item $b\ge a+c$, $a,c\ge 0$. 
\end{enumerate}
In each of the cases, one constructs an $\R_{>0}$-weighted cobordism between the corresponding $\R_{>0}$-weighted one foams. Schematically, Figure~\ref{fig9_019} shows what needs to be done in case (1) above, and similar for the other cases.  

\vspsm

\input{fig9_019}

\input{fig9_017}

\input{fig9_018}

Cobordisms between the  diagrams that replace the  corresponding vertices are shown for cases (1) and (3) in Figures~\ref{fig9_017} and~\ref{fig9_018} via  sequences of their cross-sections.  

\vspsm 

Further cases include $L(-a,b,c)$, with the first number negative (that of $L(a,b,-c)$ follows by reflection symmetry). Another case is when two numbers out of three are negative. The case $L(-a,-b,-c)$ is easy, since no modifications are done at any of the four vertices of the boundary foam. (It is likely that additional symmetries of $L(a,b,c)^{\bullet}$ can be used to reduce the number of cases but we have not checked that.)

This procedure converts $\R$-weighted 2-foam $F$ with boundary $U\cong\partial F$ to an $\R_{\ge 0}$-weighted embedded foam, denoted $F^{\circ}$, with boundary $U^{\circ}$. 

 Next, $0$-facets of $F^{\circ}$ can be removed as well, by analogy and extending our deletion of $0$-facets of the foam $U^{\circ}$. The resulting 2-foam $F^{\bullet}\subset \R^2\times [0,1)$ is  $\R_{>0}$-weighted, with the boundary $U^{\bullet}\subset \R^2\times \{0\}$. Consequently, our procedure for converting $\R$-weighted 1- and 2-foams into $\R_{>0}$-weighted 1- and 2-foams gives a homomorphism 
 \[
 \iota^{\bullet} \ : \ \Cob^{1,\up}_{\R}\lra \Cob^{1,\up}_{\R_{>0}}.
 \]
 It is clear that $\iota^{\bullet}\circ \iota=\id$, since $\iota^{\bullet}$ on foams with all facets positive is the identity map.  

 To show that 
 $\iota\circ \iota^{\bullet}=\id$ it suffices to check that $\iota$ is surjective. For that, it is enough to show that $[T(a,b)]$ is in the image of $\iota$ for all $a,b\in \R$. Consider the tripod $T(a,-b)$ for $a\ge b\ge 0$. There are two ways to merge strands $a,-b,b$ into an $a-b+b=a$ strand, with the one-vertex 2-foam cobordism connecting the two ways to merge. This translates into a cobordism between $\R$-weighted 1-foams: 
 \[
 T(a,-b)\sqcup T(a-b,b) \sim T(a,0) \sqcup T(-b,b). 
 \]
 Foam $T(a-b,b)$ has positive weights. Foam $T(a,0)$ is null-cobordant via $\R_{\ge 0}$-weighted foams, see   Figure~\ref{fig9_015}. Foam $T(-b,b)$ is null-cobordant, since $[-b,b]_{\R}=0$ and $[T(-b,b)]$ is the image of $[-b,b]_{\R}$ under the homomorphism $\tau_R$ in \eqref{eq_tau_R}. Alternatively, computation in \eqref{eq_two_aa} with $-b/2$ in place of $a$ can be converted into a description of a cobordism from $T(-b,b)$ to the empty 1-foam. Consequently, $T(a,-b)$ is cobordant via an $\R$-weighted 2-foam to an $\R_{>0}$-weighted 1-foam $T(b,a-b)$. Reflecting in the plane shows that $T(-b,a)$ is cobordant to $T(a-b,b)$. We leave the remaining cases: $T(a,-b)$ with $b>a> 0$ and $T(-a,-b)$, $a,b> 0$ to the reader.  

Consequently, $\iota$ and $\iota^{\bullet}$ are mutually-inverse isomorphisms. 
This completes the proof of Theorem~\ref{thm_iso_iota}.
\end{proof} 

Our constructions and results can be summarized into the following statement.  

\begin{theorem}\label{thm_cd_iso}
There is a commutative diagram of isomorphisms of abelian groups  
\begin{equation}
\begin{CD}
\Ztwo(\R_{>0}) @>\tau>> \Cob^{1,\up}_{\R>0}  \\
@VV{\rho}V @VV{\iota}V\\
\Ztwo(\R) @>{\tau_{\R}}>> \Cob^{1,\up}_{\R} 
\end{CD}
\end{equation}
The top arrow is given by \eqref{eq_tau}, the bottom arrow $\tau_{\R}$ is the map \eqref{eq_tau_R}.  The left arrow is the map $\rho$ in Proposition~\ref{prop_iso}, the right arrow is given by \eqref{eq_iota}.   
\end{theorem}

In particular, cobordism groups of $\R$-weighted and $\R_{>0}$-weighted planar unoriented one-foams are isomorphic, and they are isomorphic to the corresponding abelian groups generated by symbols $[a,b]$ over either all positive real  $a,b>0$ or, alternatively, all real $a,b$, subject to relations \eqref{eq_rel_a}-\eqref{eq_rel_c} in each of the two cases.

\begin{remark}
In the isomorphisms $\tau$ or  $\tau_{\R}$ in Theorem~\ref{thm_cd_iso}, commutative semigroup $\R_{>0}$ or commutative group $\R$ can be replaced by any uniquely 2-divisible commutative semigroup $H$ or by a semimodule over $\Z_{>0}[1/2]$. Unique 2-divisibility is needed to consistently split a planar $H$-weighted 1-foam into a union of tripods, since lollipop loops carry weights $a/2,b/2,(a+b)/2$. These divisions by two must exist and be consistent. One then gets an isomorphism of abelian groups 
\begin{equation}\label{eq_iso_H}
    \Cob^{1,\up}_{H} \cong \Ztwo(H). 
\end{equation}
The group $\Ztwo(H)$ can be thought of as a universal \emph{antisymmetric} 2-cocycle on $H$. Antisymmetry condition forces the bracket to be almost bilinear, see Proposition~\ref{prop_kernel}.

 \emph{Symmetric} 2-cocycles are not almost bilinear, in this sense, and allow for a richer structure. Interestingly, they naturally appear in the theory of formal groups~\cite[Section 6]{Str19}, with relations \eqref{eq_rel_c} and  $[a,b]=[b,a]$ interpreted as the infinitesimal version of the formal group law axioms. Formal groups are closely related to cobordism groups of manifolds (to the complex cobordism generalized cohomology theory).   
 
 It seems possible to interpret symmetric 2-cocycles in the framework of foam cobordisms. A step  towards such interpretation is to consider unoriented  cobordisms (2-foams with boundary) between 1-foams, not embedded anywhere, where 2-foams have oriented seams.  One imposes  the compatibility condition on seam orientations at vertices of the 2-foam to match the 2-cocycle relation. Absence of an embedding and not keeping track of the order of thin facets at a seam leads to the symmetric relation $[a,b]=[b,a]$. Antisymmetry property vanishes, since in the cobordisms in the top row of  Figure~\ref{fig6_008} the seams are now oriented and the two vertices of the boundary 1-foam for each relation have opposite types, leading us to denote the two brackets by $[a,b]_+$ and $[a,b]_-$ and giving the relation $[a,b]_+ +[b,a]_-=0$, which simply allows to express one bracket via the other. The bracket $[a,b]_+$, then, satisfies the symmetry property and the 2-cocycle condition. 
\end{remark} 

\begin{remark}
    For an interesting cobordism group we considered planar \emph{unoriented} weighted 1-foams in this section. Planar \emph{oriented} 1-foams do not allow loops and creation of tripod foams $T(a,b)$. The cobordism group of suitably defined planar oriented 1-foams is trivial.  
\end{remark}

Constructions and results of this section demonstrate the possibility of having  interesting cobordism groups of planar objects other than embedded 1-manifolds, with additional decorations, such as 
(positive) real weights.  
Notice that the resulting cobordism group has the flavour of scissor congruence groups (for instance, surjecting onto $\Lambda^2_{\Q}\R$, so that $\R$ is essentially viewed as a discrete group, which is typical of scissor congruence).  

%
%

\section{Variations on weighted foams}\label{sec_variations}

Here we go back to considering oriented $\R_{>0}$-weighted foams, not embedded anywhere, as in Section~\ref{sec_iet}. 

\subsection{Foams with flips}
\label{subsec_flips}

The group of $\IET$ automorphisms of the interval can be enhanced with flips $[0,a]\lra[0,a], x\mapsto a-x$, see~\cite{La}, viewed as a subgroup of all measurable automorphisms of $([0,1],|dx|)$.  Denote this automorphism group by $G_f$; it contains $\Aut_{\IET}$ as a subgroup. Arnoux~\cite{La} has shown that this group is simple, in particular, $[G_f,G_f]=G_f$. 

\input{fig7_001}

\input{IET_000}

\input{IET_001}

\input{fig7_002}
    
A flip of an interval $[0,a]$ can be encoded by a dot on a line labelled $a$, see Figure~\ref{fig7_001}. $\IET$ 1-foams and 2-foams can be enhanced by flip dots and flip networks, subject to the following relations: 
    \begin{itemize}
        \item Two dots on an interval can cancel via a cobordism, 
        \item A flip line on an $(a+b)$-facet can cross a seam and become two flip lines on $a,b$ facets, reversing the order of the two thin facets at the seam,   
    \end{itemize}
as shown in Figures~\ref{fig7_001} and \ref{fig7_002}. Figure~\ref{IET_000} shows how a flip on an $(a+b)$-line is modified to flips on $a$- and $b$-lines, via a concordance of braid-like 1-foams with flips.  Figure~\ref{IET_001} shows the thickened version of that equivalence transformation.  
Denote by $\CoboRf$ the cobordism group of $\R_{>0}$-decorated one-foams with flips. 

\begin{theorem}\label{thm_with_flips} The cobordism group of weighted oriented one-foams with flips is trivial, 
\begin{equation}
    \CoboRf \ \cong \ 0. 
\end{equation}
\end{theorem}
\begin{proof}
  Any  $\R_{>0}$-decorated one-foam with flips $U$ can be represented as the closure $\widehat{U}_0$ of a braid-like foam $U_0$. To $U_0$ there is associated the corresponding element $u_0\in G_f$. Since $G_f$ is perfect, $u_0$ can be represented as a product of commutators, $u_0=\prod_{i=1}^k [v_i,w_i]$. Write $U_0$ as the composition of corresponding one-foams, 
  $U_0=\prod_{i=1}^k [V_i,W_i]$. 
  The foam for each commutator is cobordant to the interval foam, using the argument as in Figure~\ref{fig7_003}. 
  Consequently $[U]=0$ in the cobordism group. 

  \vspace{0.07in}

  The theorem can also be proved directly, without invoking the perfectness of $G_f$. 
  Start with a one-foam $U$, possibly with flips.
  A dot can be split off from the rest of $U$ into an $a$-circle with a dot, see Figure~\ref{fig8_001}.  
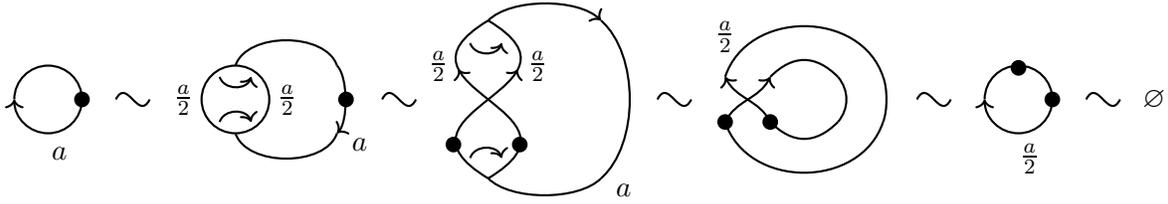
  An $a$-circle with a dot is null-cobordant, see Figure~\ref{fig8_002}.

\input{fig8_001}

\input{fig8_002}

\vspace{0.07in} 

Consequently, a 1-foam with flips is cobordant to the same foam without flips, and flips can be removed at any time when constructing a sequence of cobordisms. 
Present foam $U$ as the closure of a braid-like foam $U_0$. 
All crossings in $U_0$ can be split off from a diagram as in Figure~\ref{fig6_012}, along with the flips. 
  
An order of thin edges at a vertex can be reversed by adding three dots, one on each adjacent edge, as shown in relation (1) in Figure~\ref{fig6_005}, with an additional dot added on both sides of the relation on the $(a+b)$-line. Two dots on the $(a+b)$-line on the left hand side can then be cancelled via the Figure~\ref{fig7_001} cobordism. Dots can be split off as well and removed, being null-cobordant (Figure~\ref{fig8_002}). 

\input{fig8_003}

Combination of these moves transforms $U_0$ into a cobordant foam which is a disjoint union of foams $U_{a_i,b_i}$ and a  braid-like foam $U_2$ without crossings, dots, and compatible orders of thin edges at all vertices, see Figure~\ref{fig6_009} on the right. Foam $U_2$ is null-cobordant, as explained earlier. 
Foam $U_{a,b}$ is null-cobordant as well, as shown in Figure~\ref{fig8_003}. Consequently, $U$ is null-cobordant. 
\end{proof} 


\subsection{Foams with a map into a topological space} 
\label{subsec_foams_map}

Consider 1-foams and 2-foams equipped with a continuous map into a topological space $X$. Without loss of generality we can assume that $X$ is a connected CW-complex. One can form the abelian group $\Cob^{1,X}_{\indexw}$ of $\R_{>0}$-decorated oriented one-foams $U$ with a map $\psi:U\lra X$ modulo cobordisms. Two 1-foams as above with maps $\psi_i:U_i\lra X$ are cobordant if there is a 2-foam $F$ with a continuous map $\psi:F\lra X$
such that $\partial (F,\psi)=(U_1,\psi_1)\sqcup (-U_0,\psi_0)$. 
We can assume that $X$ is a path-connected CW-complex. 

Homotopy classes of maps $\psi:U\lra X$ of a one-foam into $X$ depend only on the fundamental group $\pi_1(X)$. Denote this group by $G$ and consider $G$-decorated one-foams, as follows. A decoration consists of finitely many dots on edges of $U$, each dot labelled by an element of $G$. Form the standard model of the classifying space $BG$, take its 2-skeleton and pass to the Poincare dual $P(G)$  of the 2-skeleton. A map of a one-foam to $BG$ can be deformed to a PL map into the one-skeleton of $BG$, which we also denote $\psi:U\lra BG^1$. Here we view a one-skeleton of $BG$ as a subspace of $P(G)$. The inverse image of the one-skeleton of $P(G)$ is then a collection of points on edges of $U$ labelled by elements of $G$. Point labelled $g\in G$ corresponds to intersections of $\psi(U)$ with the one-cell of $P(G)$ labelled $g$. 

A cobordism $F$ between 1-foams $U_1,U_2$ which is a 2-foam with a map $\psi:F\lra X$ can be converted to a PL map into the 2-skeleton $P(G)$, also denoted $\psi$. The inverse image of $P(G)^1$ is then a one-dimensional PL CW-complex in $F$ with labels on edges, with possible singularities as shown in Figure~\ref{fig8_004}.  Edges labelled $1\in G$ can be erased.  

\input{fig8_004}

\begin{prop} \label{prop_X} The cobordism group of oriented $\R_{>0}$-decorated one-foams equipped with a continuous map to a path-connected CW-complex $X$ is given by 
\begin{equation}
    (\R\otimes_{\Z} \rmH_1(X,\Z))\oplus (\R\wedge_{\Q}\R). 
\end{equation}
\end{prop} 

Note that, if $\rmH_1(X,\Z)$ is torsion, the first term vanishes. 

\begin{proof}
Denote by $[g]$ the image of $g\in G$ in $\mathrm{H}_1(G)=\mathrm{H}_1(X,\Z)$ and define a map 
    \begin{equation}
        \gamma_1\:\ \Cob^{1,X}_{\R_{>0}} \lra \R\otimes_{\Z} \rmH_1(X,\Z)
    \end{equation}
    by sending  
    a $G$-labelled oriented $1$-foam $U$ to the sum 
    \begin{equation}
        \gamma_1(U) \ := \  \sum_{i} a_i \otimes [g_i],
    \end{equation}
    where the sum is over all labels $g_i\in G$ on $U$ and $a_i$ is the thickness of the edge which contains $g_i$. 
    It is straightforward to see that $\gamma_1(U)$ depends only on the cobordism class of $U$ in $\Cob^{1,X}_{\R_{>0}}$, so it is indeed a well-defined map on cobordism classes $[U] $. Define the homomorphism
    \begin{equation}
        \gamma \ : \ \Cob^{1,X}_{\R_{>0}}\lra (\R\otimes_{\Z} \rmH_1(X,\Z))\oplus (\R\wedge_{\Q}\R), \ \ \gamma([U]) = (\gamma_1([U]), \mathsf{SAF}(U)). 
    \end{equation}
It is then straightforward to check that $\gamma$ is an isomorphism of groups. 
\end{proof}


\subsection{Other variations}
\label{subsec_other}

\begin{remark}
D.~Sullivan~\cite{Su,Ver14} has proved that any oriented one-dimensional solenoidal manifold $M$ is the boundary of an oriented solenoidal surface. The idea (tributing his much earlier conversation with B.~Edwards) is to represent $M$ as the closure of a braid-like structure, that is, as the mapping torus of a homeomorphism of the Cantor set. Sullivan then uses R.~D.~Anderson's theorem~\cite{And58} that the homeomorphism group of the Cantor set is simple and, in particular, perfect. Representing the homeomorphism as a product of commutators allows to realize $M$ as the boundary, schematically shown in Figure~\ref{fig7_003}.  
 This use of braid-like closures is analogous to that in the proofs of Theorems~\ref{thm_saf},~\ref{thm_with_flips}, where an oriented weighted 1-foam is represented as the closure of a braid 1-foam. It is likely that Sullivan's result can be interpreted as vanishing of $K_1$ of a suitable assembler category~\cite{Za1,Za2}, 
where the assembler structure is that of coverings of the Cantor set. 
\end{remark}

\begin{remark}
    The $\SAF$ invariant can be generalized to the Kenyon--Smillie invariant~\cite{KS00,Cal04}, and it is an interesting question to interpret the latter via suitably decorated foam cobordisms. 
\end{remark}

\begin{remark}
O.~Lacourte~\cite{La}  defines a version of interval exchange transformations for each subgroup $\Gamma\subset \R/\Z$, via the corresponding subgroup $\IET(\Gamma)$ of the group $\Aut_{\IET}$. He establishes an isomorphism between $H_1(\IET(\Gamma))$ and the second skew-symmetric power of $\widetilde{\Gamma}$ over $\Z$, where $\widetilde{\Gamma}$ is the preimage of $\Gamma$ in $\R$. It is straightforward to extend the results of Section~\ref{sec_iet} to interpret the above first homology group as the group of foam cobordisms, with facets of foams carrying weights in $\widetilde{\Gamma}\cap \R_{>0}$ (and see Remark~\ref{remark_H}). 

Lacorte also considers the group of $\IET$s with flips. This group is known to be perfect, and Theorem~\ref{thm_with_flips} is a foam interpretation of this result. Lacourte shows that subgroups $\IET(\Gamma)$ with flips modulo the commutator may have 2-torsion, and there should be an analogue of this result for foam cobordism.
\end{remark}



\bibliographystyle{amsalpha} 
\bibliography{k_theory}

\end{document}